\documentclass[reqno,a4paper,12pt]{amsart}

\usepackage{amsmath,amscd,amsfonts,amssymb}
\usepackage{mathrsfs,dsfont}
\numberwithin{equation}{section}

\newcommand{\define}[1]{\emph{#1}}

\def\R{\mathbb{R}}

\def\Z{\mathbb{Z}}

\def\lam{\lambda}
\def\Lam{\Lambda}
\def\1{\mathds{1}}

\renewcommand\le{\leqslant}
\renewcommand\ge{\geqslant}
\renewcommand\leq{\leqslant}
\renewcommand\geq{\geqslant}

\renewcommand\hat{\widehat}

\newcommand{\ft}[1]{\widehat #1}
\newcommand{\dotprod}[2]{\langle #1 , #2 \rangle}
\newcommand{\mes}{\operatorname{mes}}

\theoremstyle{plain}
\newtheorem{lem}{Lemma}[section]
\newtheorem{thm}[lem]{Theorem}
\newtheorem{corollary}[lem]{Corollary}

\newtheorem*{claim*} {Claim}

\newcommand{\thmref}[1]{Theorem~\ref{#1}}
\newcommand{\secref}[1]{Section~\ref{#1}}

\newcommand{\lemref}[1]{Lemma~\ref{#1}}

\newcommand{\corref}[1]{Corollary~\ref{#1}}

\theoremstyle{definition}

\newtheorem*{definition*}{Definition}
\newtheorem*{remarks*}{Remarks}
\newtheorem*{remark*}{Remark}

\newenvironment{enumerate-math}
{\begin{enumerate}
\addtolength{\itemsep}{5pt}
}
{\end{enumerate}}

\begin{document}

 \title{Spectrality and tiling by cylindric domains}

\author{Rachel Greenfeld}
\address{Department of Mathematics, Bar-Ilan University, Ramat-Gan 52900, Israel}
\email{rachelgrinf@gmail.com}

\author{Nir Lev}
\address{Department of Mathematics, Bar-Ilan University, Ramat-Gan 52900, Israel}
\email{levnir@math.biu.ac.il}

\thanks{Research partially supported by the Israel Science Foundation grant No. 225/13.}
\date{April 28, 2016}


\begin{abstract}
A bounded set $\Omega \subset \R^{d}$ is called a
spectral set if the space $L^{2}(\Omega)$ admits
a complete orthogonal system of exponential functions.
We prove that a cylindric set $\Omega$ is spectral if and only
if its base is a spectral set. A similar characterization is obtained
of the cylindric sets which can tile the space by translations.
\end{abstract}

\maketitle


\section{Introduction}

\subsection{}
Let $\Omega\subset \R^d$ be a bounded, measurable set of positive Lebesgue measure. A
discrete set $\Lambda\subset \R^d$ is called a \define{spectrum} for $\Omega$ if the
system of exponential functions
\begin{equation}
	\label{eqI1.1}
	E(\Lambda)=\{e_\lambda\}_{\lambda\in\Lambda}, \quad e_\lambda(x)=e^{2\pi
	i\dotprod{\lambda}{x}},  
\end{equation}
constitutes an orthogonal basis in $L^2(\Omega)$, that is, the system is orthogonal and
complete in the space. A set $\Omega$ which admits a spectrum $\Lambda$ is called a
\define{spectral set}.
 For example, if $\Omega$ is the unit cube in $\R^d$, then it is a
spectral set, and $\Lambda=\Z^d$ is a spectrum for $\Omega$.
\par
 The study of spectral sets was initiated in the paper \cite{Fug74} due to Fuglede (1974),
who conjectured that these
sets could be characterized geometrically in the following way: \define{the set 
$\Omega$ is spectral if and only if it can tile the space by translations}. We say that
$\Omega$ \define{tiles} the space by translations along a discrete set
$\Lambda\subset \R^d$ if the family of sets $\Omega+\lambda$ $(\lambda\in\Lambda)$
constitutes a partition of $\R^d$ up to measure zero. 
 Fuglede's conjecture inspired extensive research
over the years, and a number of interesting results supporting the conjecture had been
obtained. 
\par
For example, it was proved in \cite{Fug74} that if $\Omega$ tiles the space 
by translations along a \define{lattice}, then it is a spectral set. To the contrary, a triangle in
the plane \cite{Fug74}, or more generally, any convex non-symmetric domain in $\R^d$
\cite{Kol00a}, is not  spectral. It was also proved that the ball in $\R^d$
$(d\ge 2)$ is not a spectral set \cite{Fug74, IKP99, Fug01}, as well as any convex domain with a
smooth boundary \cite{IKT01}.  In \cite{IKT03} it was proved that a convex domain
$\Omega\subset \R^2$ is spectral if and only if it is either a parallelogram or a
centrally symmetric hexagon, which confirmed that Fuglede's conjecture is true
for convex domains in dimension $d=2$. See also the survey in \cite[Section~3]{Kol04}.
\par
On the other hand, in 2004 a counter-example to the ``spectral implies tiling''
part of the conjecture in dimensions $d\ge 5$ was found by Tao \cite{Tao04}.
Subsequently,  the ``tiling implies spectral'' part was also disproved, and the dimension in
these counter-examples (all of which are finite unions of unit cubes) was reduced up to
$d\ge 3$, see \cite[Section~4]{KM10} and the references given there. 
The conjecture is still open, though, in dimensions $d=1,2$ in both directions. 
\subsection{}
A bounded, measurable set $\Omega\subset\R^{d}$ $(d\ge 2)$ will be called a
\define{cylindric set} if it has the form
\begin{equation}
	\label{eqI1.2}
	\Omega=I\times \Sigma,
\end{equation}
where $I$ is an interval in $\R$, and $\Sigma$ is a measurable set in $\R^{d-1}$. 
In this case, the set $\Sigma$ will be called the \define{base} of the cylindric set $\Omega$.
\par
In this paper we are interested in the spectrality problem for cylindric sets. For
example, as far as we know, the following question has remained open: 
\emph{Let $\Omega$ be a cylindric set in $\R^d$ $(d\ge 3)$, 
whose base $\Sigma$ is the unit ball in $\R^{d-1}$. Is it a spectral set?}
\par
As the boundary of this set $\Omega$ is not piecewise flat (and, in particular, 
 $\Omega$ cannot tile), one would expect that the answer to this question should be negative.
However the approach in \cite{IKT01} does not apply in this situation, as it is based on
the existence of a point on the boundary of $\Omega$ where the Gaussian curvature is
non-zero, while for a cylindric set this curvature vanishes at every point where the
boundary is smooth. 
\par
The main result in this paper is the following: 
\begin{thm}
	\label{thmI1.3}
	A cylindric set $\Omega=I\times \Sigma$ is spectral (as a set in $\R^d$,
	$d\ge 2$) if and only if its base $\Sigma$ is a spectral set (as a set in
$\R^{d-1}$). 
\end{thm}
Thus we obtain a characterization of the cylindric spectral sets $\Omega$ in terms of the
spectrality of their base $\Sigma$. 
\par
In particular this confirms the negative answer
to the question stated above, due to the results in \cite{Fug74, IKP99, Fug01}. More generally,
we obtain the following:
\par
\emph{Let $\Omega$ be a cylindric set in $\R^d$ $(d\ge 3)$ whose base $\Sigma$ is a
	convex body in $\R^{d-1}$ with a smooth boundary. Then $\Omega$ is not a spectral
	set.}
\par
This follows from \thmref{thmI1.3} and the result in \cite{IKT01}. If we combine
\thmref{thmI1.3} with the result in \cite{IKT03}, it yields the following corollary in
dimension $d=3$:
\par
\emph{Let $\Omega$ be a cylindric convex body in $\R^3$. Then $\Omega$ is a spectral set
	if and only if $\Omega$ is either a parallelepiped or a centrally symmetric
	hexagonal prism.}
\par
\subsection{}
There is a commonly believed principle which states that for any result about spectral
sets there is an analogous result about sets which can tile by translations, and vice
versa. The result analogous to \thmref{thmI1.3} concerning tiling, for which we also
provide a proof in this paper, is the following: 
\begin{thm}
	\label{thmI1.6}
	A cylindric set $\Omega=I\times \Sigma$ can tile $\R^d$ $(d\ge 2)$ by
	translations if and only if its base $\Sigma$ tiles $\R^{d-1}$ by translations. 
\end{thm}
Thus we have a similar characterization of the cylindric sets which can tile.
\par
It is known, see  \cite{McM80}, that a cylindric convex body in $\R^3$ which tiles by
translations must be either a parallelepiped or a centrally
symmetric hexagonal prism (this can also be derived based on
\thmref{thmI1.6}). Hence we obtain:
\begin{corollary}
	\label{corI1.7}
	A cylindric convex body $\Omega \subset \R^3$ is a spectral set
	if and only if it can tile by translations.
\end{corollary}
In other words, Fuglede's conjecture is true for cylindric convex bodies $\Omega$ in $\R^3$.
The latter conclusion plays an important role in our paper \cite{GriLev16b}, where we establish that
Fuglede's conjecture is true for all \emph{convex polytopes}  in $\R^3$. 
\subsection{}
The ``if'' part of \thmref{thmI1.3} is obvious. Indeed,
suppose that $\Sigma$ is a spectral set, and assume for simplicity that
$I=[-\frac{1}{2},\frac{1}{2}]$. If $\Gamma\subset\R^{d-1}$ is a spectrum for $\Sigma$,
then it is easy to check that $\Lambda=\Z\times\Gamma$ is a spectrum for $\Omega$, and
hence $\Omega$ is spectral.
\par
On the other hand, the converse, ``only if'' part of the
result, is non-trivial. Roughly speaking, the difficulty lies in that knowing $\Omega$ to
have a spectrum $\Lambda$ in no way implies that $\Lambda$ has a product structure as
$\Z\times \Gamma$. In particular, we do not have any obvious candidate for a set
$\Gamma\subset\R^{d-1}$ that might serve as a spectrum for $\Sigma$.
\par
To address this difficulty we adapt (and simplify) an
approach from the paper \cite{IP98} due to Iosevich and Pedersen. The main result in that
paper is a characterization of the spectra of the unit cube in $\R^d$ by a tiling condition
(in connection with this result, see also \cite{JP99, LRW00, Kol00b}). 
The approach involves an iterative procedure, where in each step a certain
modification to the given spectrum is performed, which yields a new spectrum for $\Omega$.
This produces an infinite sequence of sets $\Lambda_n$ each one of which is a spectrum for $\Omega$. 
\par
We can prove that for an appropriately chosen series of modifications,
the sequence $\Lambda_n$ converges weakly to a limit $\Lambda'$
which is also a spectrum for $\Omega$, and which satisfies the additional condition that 
\[ \Lambda'\subset \Z\times \R^{d-1}. \]
It is then possible to use a result due to Jorgensen and Pedersen \cite{JP99} which yields that the
cylinder's base $\Sigma$ must be a spectral set, and thus we obtain \thmref{thmI1.3}.
\par
 The proof of \thmref{thmI1.6} is given along similar lines.


\section{Spectrality and Tiling}
We start by recalling some basic properties of spectra and tilings that will be used in the next sections. 

\subsection{}

Let $\Omega\subset\R^d$ be a bounded, measurable set of positive measure. A discrete set
$\Lambda\subset\R^d$ is called a spectrum for $\Omega$ if the system of exponential functions
$E(\Lambda)$ defined by \eqref{eqI1.1} is an orthogonal basis in the space $L^2(\Omega)$.
\par
For any two points $\lam,\lam'$ in $\R^d$ we have
\[
		\dotprod{e_\lambda}{e_{\lambda'}}_{L^2(\Omega)} 
= \hat{\1}_\Omega(\lambda'-\lambda), 
\]
where
\[
\hat{\1}_\Omega(\xi) = \int_{\Omega} e^{-2\pi i\langle \xi,x\rangle} dx, \quad \xi \in \R^d
\]
is the Fourier transform of the indicator function $\1_\Omega$ of the set $\Omega$.
It follows that the orthogonality of the system $E(\Lambda)$ in $L^2(\Omega)$ is equivalent to the
condition
\begin{equation}
	\label{eqP1.2}
	\Lambda-\Lambda\subset \{\hat{\1}_\Omega=0\} \cup \{0\}.
\end{equation}

\subsection{}
Let $f \geq 0$ be a measurable function on $\R^d$. We say that $f$
 \emph{tiles}  $\R^d$ by  translations along a discrete set
$\Lambda\subset\R^d$ if we have
\begin{equation}
\label{1.1}
\sum_{\lambda\in\Lambda}f(x-\lambda)=1\quad\text{a.e.}
\end{equation}
In this case we will shortly write that  $f+\Lam$ is a \emph{tiling}.
\par
If $f=\1_\Omega$ is the indicator function of a bounded, measurable set
$\Omega\subset\R^d$, then the condition \eqref{1.1}
means that the family of sets $\Omega+\lam$ $(\lam\in\Lam)$ constitutes
a partition of $\R^d$ up to measure zero. In this case, we will say that  $\Omega+\Lam$  is a tiling.

\subsection{}
A set $\Lambda\subset \R^d$ is said to be \define{uniformly discrete} if there is
$\delta>0$ such that $|\lambda'-\lambda|\ge \delta$ for any two distinct points
$\lambda,\lambda'$ in $\Lambda$. The maximal constant $\delta$ with this property is
called the \define{separation constant} of $\Lambda$, and will be denoted by
$\delta(\Lambda)$.
\par
 The condition \eqref{eqP1.2} implies that if $\Lambda$ is a spectrum
for $\Omega$ then it is a uniformly discrete set, with separation constant
$\delta(\Lambda)$ which is not smaller than 
\begin{equation}
	\label{eqP1.4}
	\chi(\Omega):= \min \big\{ |\xi| \;:\; \xi\in\R^d, \;\; \hat{\1}_\Omega(\xi)=0
	\big\}> 0.
\end{equation}
\par
Also if $\Omega+\Lam$ is a tiling then the  set $\Lam$ must be uniformly discrete,
and in this case the separation constant $\delta(\Lam)$ is not less than
\begin{equation}
	\label{eqP1.5}
	\eta(\Omega):=\min  \big\{ |x| \;:\;  x\in\R^d, \;\; \mes (\Omega \cap (\Omega+x)) = 0 \big\} > 0.
\end{equation}
This is due to the sets $\Omega + \lam$ $(\lam \in \Lam)$ being pairwise disjoint up to measure zero.

\subsection{}
We denote by $|\Omega|$ the Lebesgue measure of the set $\Omega$.
The following lemma gives a characterization of the spectra of $\Omega$ by a tiling condition:
\begin{lem}[\cite{Kol00b}]
	\label{lem:tilespec}
	Let $\Omega\subset\R^d$ be a bounded, measurable set, and define  \[f(x):= |\ft{\1_\Omega}(x)|^2 / |\Omega|^2,\quad x\in\R^d.\]
	Then  a set $\Lam\subset\R^d$ is a spectrum for $\Omega$ if and only if $f+\Lam$ is a tiling.
\end{lem}
\begin{proof}
	Assume first that $\Lam$ is a spectrum for $\Omega$, so the system $E(\Lam)$ is orthogonal and complete in $L^2(\Omega)$. Hence by Parseval's equality
	\begin{equation}
	\label{eq:Parseval}
	\sum_{\lam\in\Lam}|\langle  e_\lam,g\rangle |^2=|\Omega|\cdot\left\| g\right\| ^2
	\end{equation}
	for every $g\in L^2(\Omega)$. In particular, using \eqref{eq:Parseval} for $g=e_x$ yields 
	\begin{equation}
	\label{eq:tilebyLam}
	\sum_{\lam\in\Lam}|\ft{\1}_\Omega(x-\lam)|^2=|\Omega|^2,
	\end{equation}
	for every $x\in\R^d$. Hence $f+\Lam$ is a tiling.
	
	Conversely, suppose that \eqref{eq:tilebyLam} holds for almost every $x\in\R^d$. Since $\ft{\1}_\Omega$ is a continuous function, the left-hand side of \eqref{eq:tilebyLam} must be \emph{everywhere} not greater than $|\Omega|^2$. Using this with $x$ going through the elements of $\Lam$ yields that $E(\Lam)$ is an orthogonal system in $L^2(\Omega)$. Moreover, the system $E(\Lam)$ spans every exponential $e_x$ for which the equality in \eqref{eq:tilebyLam} is satisfied.
	Since this holds for a dense set of points $x$ in $\R^d$, the system $E(\Lam)$ is complete in $L^2(\Omega)$. Hence $\Lam$ is a spectrum for $\Omega$.
\end{proof}


\section{Limits of spectra and tilings}\label{sec:Weak-L}

\subsection{}
Let $\Lambda_n$ be a sequence of uniformly discrete sets in $\R^d$, with separation constants
$\delta(\Lambda_n)\ge\delta>0$. The sequence $\Lambda_n$ is said to \define{converge
weakly} to a set $\Lambda$ if for every $\varepsilon>0$ and every $R$ there is
$N$ such that 
\[
\Lambda_n\cap B_R\subset \Lambda+B_\varepsilon \quad \text{and} \quad \Lambda\cap B_R\subset
\Lambda_n+B_\varepsilon
\]
 for all $n\ge N$, where by $B_r$ we denote the open ball of
radius $r$ centered at the origin.
\par
 In this case the weak limit $\Lambda$ is also 
uniformly discrete, and moreover, $\delta(\Lambda)\ge \delta$. 
\begin{lem}
	\label{lem:limtile}
	Let $ \Lambda_n$ be a sequence  of uniformly discrete sets in $\R^d$, $\delta(\Lam_n)\geq\delta>0$, which converges weakly to a set $\Lambda$. Suppose that  $f\in L^1(\R^d)$, $f\geq 0$, and that  $f+\Lam_n$ is a tiling for every $n$.  Then also $f+\Lam$ is a tiling.
\end{lem}
\begin{proof}
It would be enough to show that if $\varphi\geq 0$ is a smooth, compactly supported function on $\R^d$,  $ \int \varphi(x) dx =1$, then 
 	\begin{equation}
 	\label{eq:enoughcondition}
 \int_{\R^d} \varphi(x)\sum_{\lam\in\Lam} f(x-\lam)\,dx =1.
 	\end{equation}
Fix such a function $\varphi$, and define 
\begin{equation}
\label{eq:definePhi}
\Phi_{n}(x):=\sum_{\lam\in\Lam_n}\varphi(x+\lam).
\end{equation} 
The weak convergence of $\Lam_n$ to $\Lam$ implies that for any smooth, compactly supported function $\psi$ on $\R^d$ we have 
\begin{equation}
\label{eq:lim}
	\lim_{n\to\infty} \int_{\R^d} \psi(x)\Phi_{n}(x)dx=\int_{\R^d}\psi(x)\Phi(x)dx,
\end{equation}
where $\Phi(x)$ is defined as in \eqref{eq:definePhi} but with $\Lam$ instead of $\Lam_n$. 
Moreover, we have $\|\Phi_{n}\|_\infty\leq C$, where the constant $C=C(\varphi,\delta)$
 does not depend on $n$.
Hence $\Phi_{n}$ converges weakly in $L^\infty(\R^d)$ to $\Phi$.
In particular, \eqref{eq:lim} is satisfied also for the function $\psi=f$ which is
in $L^1(\R^d)$.
But since $f+\Lam_n$ is a tiling, the left-hand side of \eqref{eq:lim} in this case is equal to $1$.
It follows that also the right-hand side of \eqref{eq:lim} must be $1$, which implies \eqref{eq:enoughcondition}.
\end{proof}

\subsection{}
By applying  \lemref{lem:limtile} to the function $f=|\ft{\1}_\Omega|^2/|\Omega|^2$ and to
$f=\1_\Omega$, we obtain:
\begin{corollary}\label{corB1.1}
	Let $\Omega\subset\R^d$ be a bounded, measurable set.
\begin{enumerate-math}
	\item
	Suppose that for each $n$, the set $\Lam_n$ is a spectrum for $\Omega$. 
	If $\Lam_n$ converges weakly to $\Lam$, then also $\Lam$ is a spectrum for $\Omega$.
	\item
	Suppose that $\Omega+\Lam_n$ is a tiling for every $n$.
	If $\Lam_n$ converges weakly to  $\Lam$,  then also $\Omega+\Lam$ is a tiling.
\end{enumerate-math}
\end{corollary}

\begin{remark*}
Using weak limits is a well-known technique in the analysis of frames and Riesz systems
of exponentials, which in this context goes back to Beurling. 
To our knowledge, so far this technique has not been used in the study of spectral sets.
\end{remark*}


\section{Orthogonality and Packing}

\subsection{}
If $f\geq0$ is a  measurable function on $\R^d$, and
$\Lam$ is a discrete set in $\R^d$, then following \cite{Kol00b} we say that $f+\Lam$ is a \emph{packing} if we have
\[
\sum_{\lam\in\Lam}f(x-\lam)\leq 1\quad\text{a.e.}
\]
\par
Notice that if $f=\1_\Omega$ then this just means that the sets $\Omega+\lam$ $(\lam\in\Lam)$
are pairwise disjoint up to measure zero. In this case we will say that $\Omega+\Lam$
is a packing.

\subsection{}\label{secAB}
Let $A,B$ be two discrete sets in $\R^d$, and $\tau$ be a vector in $\R^d$.
Suppose that each one of the three sets $B$, $B+\tau$, $B-\tau$ is disjoint from $A$,
and define
\[\Lam:=A\cup B,\quad\Lam':=A\cup (B+\tau),\quad \Lam'':=A\cup(B-\tau). \]
\begin{lem}
	\label{lemM2.1}
Let $f \geq 0$ be a measurable function on $\R^d$.
	Assume that $f+\Lam$ is a tiling, while $f+\Lam'$ and $f+\Lam''$ are packings. 
Then  $f+\Lam'$ and  $f+\Lam''$ are both tilings.
\end{lem}
\begin{proof}
	Define
	\[ f_A(x):=\sum_{\lam\in A}f(x-\lam), \quad f_B(x):=\sum_{\lam\in B}f(x-\lam). \]
	Then the assumptions of the lemma mean that
\begin{enumerate-math}
	\item
		\label{eq:assumption1}
		$ f_A(x)+f_B(x)=1,$
		\item
		\label{eq:assumption2}
		$f_A(x)+f_B(x-\tau)\leq 1,$
		\item
		\label{eq:assumption3}
		$  f_A(x)+f_B(x+\tau)\leq 1$
\end{enumerate-math}
for a.e.\ $x \in \R^d$. Subtracting \ref{eq:assumption1} from each one of
\ref{eq:assumption2} and \ref{eq:assumption3} yields
	\begin{equation}
	\label{eq:conclude1}
f_B(x-\tau)\leq f_B(x), \quad 
f_B(x+\tau)\leq f_B(x)\quad \text{a.e.}
	\end{equation}
Translating by the vector $\tau$ one can then
see that \eqref{eq:conclude1} is only possible if
	 \[ f_B(x-\tau)=f_B(x)=f_B(x+\tau) \quad \text{a.e.},\] 
	 which implies the assertion of the lemma.
\end{proof}

\subsection{}
	Let $\Omega$ be a bounded, measurable set in $\R^d$.
The following lemma gives a criterion for the orthogonality of an exponential
system in $L^2(\Omega)$ by a packing condition:
\begin{lem}[\cite{Kol00b}]
	\label{lem:orthpack}
	A system of exponential functions $E(\Lam)$ is  orthogonal in $L^2(\Omega)$ if and only if  $f+\Lam$ is a packing,
	where $f:=|\ft{\1}_\Omega|^2/|\Omega|^2$.
\end{lem}
This can be proved  in a similar way to \lemref{lem:tilespec}.
\par
Combining Lemmas \ref{lem:tilespec}, \ref{lemM2.1} and \ref{lem:orthpack} we obtain:
\begin{corollary}
	\label{corM2.3}
Under the assumptions in \secref{secAB}, if the set $\Lam$ is a spectrum for $\Omega$,
and if both systems $E(\Lam')$ and $E(\Lam'')$  are orthogonal in $L^2(\Omega)$,
then these systems are also complete in $L^2(\Omega)$, that is, each one of the sets
$\Lam'$ and $\Lam''$ is a spectrum for $\Omega$.
\end{corollary}


\section{Cylindric sets}

In this section we assume  $\Omega$ to be a cylindric set in $\R^d$, namely
\begin{equation}
	\label{eqOC1}
 \Omega=I\times \Sigma
\end{equation}
where $I\subset\R$ is an interval, and  $\Sigma$ is a bounded, measurable  set in $\R^{d-1}$.
\par
Moreover, we will assume for simplicity that $I=[-\frac{1}{2},\frac{1}{2}]$. In the next section
we will reduce the general situation to this more specific one by applying an affine transformation,
so this assumption will not result in any loss of generality.
\par
We shall denote a point $x\in\R^d$ as $x=(x_1,x_2)$, where $x_1\in\R$ and $x_2\in\R^{d-1}$.

\subsection{}
Due to \eqref{eqOC1} the Fourier transform of the indicator function $\1_\Omega$ is given by
 \begin{equation}
\label{eq:LamtOG}
\ft{\1}_\Omega(\xi)=\ft{\1}_I(\xi_1)\, \ft{\1}_\Sigma(\xi_2), \quad \xi = (\xi_1, \xi_2) \in \R \times \R^{d-1}.
\end{equation}
Since the zero set of the function $\hat{\1}_I$ is $\Z \setminus \{0\}$, it is easy to confirm the following:
\begin{lem}
\label{lemM1.1}
Let $\lam, \lam'$ be two points in $\R^d$.
The exponentials $e_{\lam}$ and $e_{\lam'}$ are orthogonal in $L^2(\Omega)$ if and only if
$\lam_1'-\lam_1$ is a non-zero integer, or $ \lam'_2 - \lam_2$ lies in the zero set of the function
$\ft{\1}_\Sigma$.
\end{lem}

\subsection{}
We also have the following parallel statement for tiling by translations of the cylindric set $\Omega$, although
it is not completely analogous to \lemref{lemM1.1}.
\begin{lem}
	\label{lemM1.2}
	Assume that $\Omega+\Lam$ is a tiling, and let $\lam,\lam'$
 be two distinct points in $\Lam$.  Then $\lam_1'-\lam_1$ is a non-zero integer,
or the set $( \Sigma+\lam_2)\cap(\Sigma+\lam_2')$  has  measure zero.
\end{lem}	
\begin{proof}
Suppose that  the set $( \Sigma+\lam_2)\cap(\Sigma+\lam_2')$
	in $\R^{d-1}$ has  positive  measure. We will show that in this case
$\lam_1'-\lam_1$ must be a non-zero integer.
\par
Observe first that we must have $|\lam_1'-\lam_1|\geq 1$, for otherwise the set $( \Omega+\lam)\cap(\Omega+\lam')$ in $\R^d$ would have positive measure, which contradicts the assumption that $\Omega+\Lam$ is a tiling. 
	By symmetry, we can therefore assume that $\lam'_1-\lam_1 \geq 1$.
\par
	The proof  is by induction on the value of the smallest integer $n$ such that $\lam'_1-\lam_1 \leq n$.
If $n=1$ then it means that $\lam_1'-\lam_1= 1$, so in this case the assertion is true.
\par
Now consider the case when $n>1$. Then we have $\lam'_1-\lam_1> 1$. Hence, if we 
consider the cylindric set in $\R^d$ defined by
	\[S:=\left[ \lam_1+\tfrac{1}{2},\lam_1'-\tfrac{1}{2}\right] \times \left( (\Sigma+\lam_2)\cap(\Sigma+\lam'_2)\right), \]
	then  $S$ has positive measure. Since $\Omega+\Lam$ is a tiling, there must therefore exist some $\lam''\in\Lam$ such that the set 
$(\Omega+\lam'')\cap S$ is of positive measure.
	Notice that for $\Omega+\lam''$ to intersect $S$ with positive measure, it is necessary and sufficient that 
	\[\lam_1<\lam''_1<\lam'_1\]
	and that the set
	\[(\Sigma+\lam_2)\cap (\Sigma+\lam'_2)\cap(\Sigma+\lam''_2) \] 
	in $\R^{d-1}$ has positive measure.
\par
However, since  the point $\lam''$ is different from both $\lam$ and $\lam'$, and since $\Omega+\Lam$ is a tiling, 
the set $\Omega+\lam''$ can intersect neither $\Omega+\lam$ nor $\Omega+\lam'$ with positive measure.
This implies that we must actually have
	\[ \lam_1+1\leq\lam''_1\leq\lam'_1-1 . \]
	We conclude that both $\lam''_1-\lam_1$ and $\lam'_1-\lam''_1$ cannot be greater than $n-1$.
	So by the inductive hypothesis it  follows that $\lam''_1-\lam_1$ and $\lam'_1-\lam''_1$ are both integers. 
	Thus also $\lam'_1-\lam_1$ must be a (non-zero) integer, as we had to show.
\end{proof}

\subsection{}\label{secC1.3}
Let $\Lam$ be a discrete set in $\R^d$. Given $t\in \R$ we consider a mapping $\alpha_t$ defined by
\[\alpha_t ( \lambda) := 
\begin{cases}
\lambda, & \lambda_1\in\Z,\\
\lambda+ \tau(t), & \lambda_1\notin\Z
\end{cases} 
\]
for each $\lam\in\Lam$, where $\tau(t):=(t,0,0,\dots,0)  \in\R^{d}$.
\begin{samepage}
\begin{lem}
\label{lemC1.4}
\quad
  	\begin{enumerate-math}
  		\item
  		\label{lem:alphat1}
  		Assume that $\Lam$ is a spectrum for $\Omega$. Then
  		 $\alpha_t$ is a one-to-one mapping on $\Lam$, and its image $\alpha_t(\Lam)$ is also a spectrum for $\Omega$.
	\item
  		\label{lem:alphat2}
Similarly, if $\Omega+\Lam$ is a tiling, then again 
  		 $\alpha_t$ is one-to-one on $\Lam$, and $\Omega+\alpha_t(\Lam)$ is also a tiling.
  	\end{enumerate-math}
  \end{lem}	
 \end{samepage}
This is a variant of \cite[Lemma 3.3]{IP98} where a similar result was proved
for spectra and tilings by the unit cube. Here we give an alternative proof of this lemma,
based on the results obtained above.

  \begin{proof}[Proof of \lemref{lemC1.4}]
	First we prove part \ref{lem:alphat1}.
	  Assume that $\Lambda$ is a spectrum for $\Omega$. Using 
	  \lemref{lemM1.1} one can see that the exponentials  $e_{\alpha_t(\lambda)}$ and 
	  $e_{\alpha_t(\lambda')}$ are orthogonal in $L^2(\Omega)$ whenever
	  $\lambda,\lambda'$ are two distinct points in $\Lambda$. Hence $\alpha_t$ is a
	  one-to-one mapping on $\Lambda$, and the system of exponentials
	  $E (\alpha_t(\Lambda) )$ is orthogonal in $L^2(\Omega)$.
\par
In a similar way, the
	  same is true also for the mapping $\alpha_{-t}$. 
\par
Consider a partition of $\Lambda$
	  into two disjoint sets 
	  \begin{equation}
		  \label{eqP1.A}
		  A:= \left\{ \lambda\in\Lambda\; :\; \lambda_1\in\Z \right\}, \quad 
		  B:= \left\{ \lambda\in \Lambda\; :\; \lambda_1\not\in \Z \right\}.
	  \end{equation}
	  Then we have 
	  \[ \alpha_t(\Lambda)= A\cup (B+\tau(t)), \quad \alpha_{-t}(\Lambda)= A\cup (B-\tau(t)). \]
	  Hence we may apply \corref{corM2.3}, which yields that each one of the  sets
	  $\alpha_{t}(\Lambda)$ and $\alpha_{-t}(\Lambda)$ is a spectrum for $\Omega$.
	This proves part \ref{lem:alphat1}.
\par
 Now	we turn to prove part \ref{lem:alphat2}.
	  Suppose that $\Omega+\Lambda$ is a tiling. Due to  \lemref{lemM1.2},
	 if $\lambda,\lambda'$ are two
	  distinct points in $\Lambda$, then $\lambda_1'-\lambda_1$
	  is a non-zero integer or $(\Sigma+\lambda_2)\cap (\Sigma+\lambda'_2)$ is a set
	  of measure zero in $\R^{d-1}$. In either case it follows that the intersection
	  of the two sets $\Omega+\alpha_t(\lambda)$ and $\Omega+\alpha_t(\lambda')$ has
	  measure zero. Hence, $\alpha_t$ is a one-to-one mapping on $\Lambda$, and
	  $\Omega+\alpha_t(\Lambda)$ is a packing. The same is true also for the mapping
	  $\alpha_{-t}$. 
\par
Consider again the partition of $\Lambda$ given by \eqref{eqP1.A}. This time we apply
	  \lemref{lemM2.1} with the function $f=\1_{\Omega}$, which implies that
	  $\Omega+\alpha_t(\Lambda)$ and $\Omega+\alpha_{-t}(\Lambda)$ are both tilings. 
	Hence part \ref{lem:alphat2} is also proved.
  \end{proof}


\subsection{}
\begin{samepage}
\begin{lem}
	\label{lemC2.1} 
\quad 
	\begin{enumerate-math}
	\item \label{lemC2.1_1} Suppose that $\Omega$ is a spectral set. Then
		$\Omega$ admits a spectrum $\Lambda$ satisfying 
		\begin{equation}
			\label{eqC2.1.1}
			\Lambda \subset \Z\times \R^{d-1}.
		\end{equation}
	\item \label{lemC2.1_2} Similarly, if $\Omega$ can tile by translations,
		then there is a set $\Lambda$ satisfying \eqref{eqC2.1.1} such
		that $\Omega+\Lambda$ is a tiling.
	\end{enumerate-math}
\end{lem}
\end{samepage}
\begin{proof}
	Suppose first that $\Omega$ is spectral, and let $\Lambda$ be a spectrum for
	$\Omega$. Let $\delta:= \chi(\Sigma)/3$, where the constant $\chi(\Sigma)$ is defined as in 
	\eqref{eqP1.4}. Choose a sequence $\{y_n\}\subset \R^{d-1}$ satisfying 
	\begin{equation}
		\label{eqC2.1.2.0}
	|y_n|\to \infty, \quad n\to\infty,
	\end{equation}
	and such that 
	\begin{equation}
		\label{eqC2.1.2}
		\bigcup_{n=1}^\infty U_n = \R^{d-1}
	\end{equation}
	where $U_n$  denotes the open ball in $\R^{d-1}$ of radius $\delta$ centered at the point $y_n$.
\par
We define by induction a sequence of sets $\Lambda_n$, each one of which
	is a spectrum for $\Omega$, in the following way. Let $\Lambda_0:=\Lambda$. Now
	suppose that the sets 
	$\Lambda_0, \Lambda_1, \dots , \Lambda_{n-1}$
	have already been defined. Since $\Lambda_{n-1}$ is a spectrum for $\Omega$, then
	by \lemref{lemM1.1} there is a number $t_n$, $0\le t_n<1$, such that 
	\begin{equation}
		\label{eqC2.1.5}
		\Lambda_{n-1}\cap (\R\times U_n)\subset (\Z-t_n)\times U_n. 
	\end{equation}
	Then we define 
	\[ \Lambda_n:=\alpha_{t_n}(\Lambda_{n-1}), \]
	where $\alpha_{t_n}$ is the mapping from \secref{secC1.3}. It follows from
	\lemref{lemC1.4} that also $\Lambda_n$ is a spectrum for $\Omega$.
\par
 Due to the
	choice of the number $t_n$ at the $n$'th step of the construction, and since the
	mapping $\alpha_{t_n}$ leaves fixed all the points belonging to $\Z\times
	\R^{d-1}$, it follows that 
	\begin{equation}
		\label{eqC2.1.3}
		\Lambda_n\cap (\R \times (U_1\cup\dots\cup U_n)) \subset \Z\times
		\R^{d-1}
	\end{equation}
	for every $n$. Moreover, by \eqref{eqC2.1.2.0} and \eqref{eqC2.1.2}, for any $R$ there is $N$ such that 
	\begin{equation}
		\label{eqC2.1.4}
		\Lambda_n\cap (\R\times B_R) = \Lambda_m\cap (\R\times B_R), \quad n,m \geq N,
	\end{equation}
	where  $B_R$  denotes the open ball in $\R^{d-1}$ of radius $R$ centered at the
	origin. The latter fact implies that the sequence $\Lambda_n$
	converges weakly to a certain set $\Lambda'$, which by \corref{corB1.1} is also a
	spectrum for $\Omega$. It follows from \eqref{eqC2.1.2} and \eqref{eqC2.1.3} that 
the new spectrum $\Lam'$ satisfies
	\begin{equation}
		\label{eqC2.1.6}
		\Lambda'\subset \Z\times \R^{d-1},
	\end{equation}
	which establishes part \ref{lemC2.1_1} of the lemma. 
\par
The proof of part
	\ref{lemC2.1_2} is along the same line. Assume that $\Omega+\Lambda$ is a tiling.
	We choose a sequence $\{y_n\}$ with the same properties, but for
	$\delta:=\eta(\Sigma)/3$, where $\eta(\Sigma)$ is defined as in \eqref{eqP1.5}. Then
	the construction is performed in the same way, where the existence of the number
	$t_n$ satisfying \eqref{eqC2.1.5} at the $n$'th step of the construction is now
	guaranteed by \lemref{lemM1.2}. Again we obtain a sequence $\Lambda_{n}$, and
	$\Omega+\Lambda_n$ is a tiling for every $n$ (\lemref{lemC1.4}). As before, the
	sequence $\Lambda_n$ converges weakly to a limit $\Lambda'$ satisfying
	\eqref{eqC2.1.6}, and $\Omega+\Lambda'$ is a tiling by \corref{corB1.1}. Thus
	part \ref{lemC2.1_2} is also proved.
\end{proof}

\subsection{} 
\begin{samepage}
\begin{lem}	
\label{lemC2.2}
\quad
\begin{enumerate-math}
\item \label{lem:Sigma2.1} 
Suppose that $\Omega$ admits a spectrum $\Lambda \subset \Z\times \R^{d-1}$. Then each one of the sets 
\begin{equation}\label{eqC2.1.1a}
\Gamma_k:=\{\gamma\in\R^{d-1} \,:\, (k,\gamma)\in\Lam\},\quad k\in\Z ,
\end{equation} 
constitutes a spectrum for  $\Sigma$.
\item \label{lem:Sigma2.2}
Similarly, if $\Omega+\Lam$ is a tiling and the set $\Lam$ is contained in $\Z\times\R^{d-1}$,
then $\Sigma+\Gamma_k$ is a tiling for each one of the sets $\Gamma_k$ defined by \eqref{eqC2.1.1a}.
\end{enumerate-math}
\end{lem}
\end{samepage}
Part \ref{lem:Sigma2.2} of this lemma is obvious, so we shall skip its proof. Part
\ref{lem:Sigma2.1} is a consequence of \cite[Lemma 2]{JP99}. For  completeness 
we include a self-contained proof of \ref{lem:Sigma2.1}.
\begin{proof}[Proof of part \ref{lem:Sigma2.1} of \lemref{lemC2.2}]
	Let $\Lambda \subset \Z\times \R^{d-1}$ be a spectrum for $\Omega$. Fix $k \in \Z$.
	 Observe that by \lemref{lemM1.1}, 
	if $\gamma, \gamma'$ are two distinct elements of $\Gamma_k$, then
	$\gamma'-\gamma$ must lie in the zero set of $\hat{\1}_\Sigma$. Hence the system
	$E(\Gamma_k)$ is orthogonal in $L^2(\Sigma)$.
\par
 It remains to prove that this system
	is also complete in $L^2(\Sigma)$. Suppose that this is not true. Then there is
	$f\in L^2(\Sigma)$ not identically zero a.e., such that 
	\begin{equation}
		\label{eqC2.1.4a}
		\dotprod{f}{e_\gamma}_{L^2(\Sigma)}=0, \quad \gamma\in\Gamma_k.
	\end{equation}
	Consider a function $F$ defined on $\Omega = I\times \Sigma$ by 
	\[ F(x,y):=e_k(x) f(y), \quad (x,y)\in I\times \Sigma. \]
	Then $F\in L^2(\Omega)$. We claim that $F$ is orthogonal in $L^2(\Omega)$ to all
	the elements of the system $E(\Lambda)$. Indeed, we have $\Lambda\subset \Z\times
	\R^{d-1}$, hence if $\lambda$ is a point in $\Lambda$, then it has the form
	$\lambda=(m,\gamma)$, where $m\in\Z$ and $\gamma\in \Gamma_m$. This implies that 
	\begin{equation}
		\label{eqC2.1.3a}
		\dotprod{F}{e_{\lambda}}_{L^2(\Omega)}=\dotprod{e_k}{e_m}_{L^{2}(I)}\cdot
		\dotprod{f}{e_\gamma}_{L^2(\Sigma)}.
	\end{equation}
	If $m\ne k$ then the first term on the right-hand side of \eqref{eqC2.1.3a}
	vanishes, while if $m=k$ then the second term must vanish due to
	\eqref{eqC2.1.4a}. This confirms that 
	\[\dotprod{F}{e_\lambda}_{L^2(\Omega)}=0, \quad \lambda\in\Lambda. \]
	But since $F$ does not vanish identically a.e., this contradicts the
	completeness of the system $E(\Lambda)$ in $L^2(\Omega)$. This contradiction
	concludes the proof. 
\end{proof}


\section{Conclusion of Theorems \ref{thmI1.3} and \ref{thmI1.6} } \label{secD1}
The main results now follow easily from the previous lemmas.
\par
  Let $\Omega = I \times \Sigma$ be a cylindric set in $\R^d$. By applying an affine 
transformation we may assume that $I=[-\frac{1}{2}, \frac{1}{2}]$
(it is well-known and easy to verify that the family of spectral sets, or the sets
which tile by translations, is invariant under invertible affine transformations).
\par
 If $\Omega$ is spectral, then by \lemref{lemC2.1} it
admits a spectrum $\Lambda\subset \Z\times \R^{d-1}$. Then \lemref{lemC2.2} implies that
$\Sigma$ is a spectral set. Conversely, assume that $\Sigma$ is a spectral set, and let
$\Gamma\subset \R^{d-1}$ be a spectrum for $\Sigma$. It is then easy to verify that
$\Lambda=\Z\times \Gamma$ is a spectrum for $\Omega$ (see, for example,
	\cite[Theorem 3]{JP99}), and hence $\Omega$ is also spectral. Thus \thmref{thmI1.3} is
	established. 
\par
The conclusion of the proof of \thmref{thmI1.6} is along the same line.


\section{Remark} \label{secR1}
By applying \thmref{thmI1.3} (respectively, \thmref{thmI1.6}) several times, we obtain the
following more general version of the results:
\begin{thm}
	\label{thmR1.1}
	Let $\Omega=Q\times \Sigma$, where $Q$ is a cube in $\R^n$, and $\Sigma$ is
	a bounded, measurable set in $\R^m$ $(n,m\ge 1)$. Then the set $\Omega \subset \R^{n+m}$ is spectral
	if and only if $\Sigma$ is a spectral set. Analogously, $\Omega$ tiles by translations
	if and only if $\Sigma$ tiles.
\end{thm}

\subsection*{Note added in proof}
After the submission of this paper, we were informed by M.\ Kolountzakis that the 
``only if'' part of \thmref{thmI1.6} can be proved in the following more general form:
\emph{Let $A \subset \R^n$, $B \subset \R^m$ be two bounded, measurable sets.
If the set $\Omega=A \times B \subset \R^{n+m}$ can tile by translations,
then the same is true for both $A$ and $B$.}
It is not known whether the analogous assertion for spectral sets is also true.




\begin{thebibliography}{99}
	
\bibitem[Fug74]{Fug74}
	B. Fuglede, Commuting self-adjoint partial differential operators and a group
	theoretic problem. J. Funct. Anal. \textbf{16} (1974), 101--121.
	
\bibitem[Fug01]{Fug01}
	B. Fuglede, Orthogonal exponentials on the ball. Expo. Math. \textbf{19} (2001), no. 3, 267--272.

\bibitem[GL16]{GriLev16b}
	R. Greenfeld, N. Lev, Fuglede's spectral set conjecture for convex polytopes. Preprint (2016), \texttt{arXiv:1602.08854}.

\bibitem[IKP99]{IKP99}
	A. Iosevich, N. Katz, S. Pedersen, 
	Fourier bases and a distance problem of Erd\H{o}s. Math. Res. Lett. \textbf{6} (1999), no. 2, 251--255.

\bibitem[IKT01]{IKT01}
	A. Iosevich, N. Katz, T. Tao, Convex bodies with a point of
	curvature do not have Fourier bases. Amer. J. Math. \textbf{123} (2001), no. 1, 115--120.

 \bibitem[IKT03]{IKT03}
	 A. Iosevich, N. Katz, T. Tao, The Fuglede spectral conjecture holds
	for convex planar domains. Math. Res. Lett. \textbf{10} (2003), no. 5--6, 559--569.

 \bibitem[IP98]{IP98}
	A. Iosevich, S. Pedersen, 
	Spectral and tiling properties of the unit cube. 
	Int. Math. Res. Not. IMRN (1998), no. 16, 819--828. 
	 
 \bibitem[JP99]{JP99}
	 P. Jorgensen, S. Pedersen, Spectral pairs in Cartesian coordinates.
	 J. Fourier Anal. Appl. \textbf{5} (1999), no. 4, 285--302. 

\bibitem[Kol00a]{Kol00a}
	M. Kolountzakis, Non-symmetric convex domains have no basis of
	exponentials. Illinois J. Math. \textbf{44} (2000), no. 3, 542--550. 

\bibitem[Kol00b]{Kol00b}
	M. Kolountzakis, Packing, tiling, orthogonality and completeness. 
	Bull. Lond. Math. Soc. \textbf{32} (2000), no. 5, 589--599.

\bibitem[Kol04]{Kol04}
	M. Kolountzakis, The study of translational tiling with Fourier analysis.
	Fourier analysis and convexity, pp.\ 131--187, Birkh\"auser, 2004.

\bibitem[KM10]{KM10}
	  M. Kolountzakis, M. Matolcsi, Teselaciones por traslaci\'{o}n (Spanish).  Gac. R. Soc. Mat.
	 Esp. \textbf{13} (2010), no. 4, 725--746. English version in \texttt{arXiv:1009.3799}.

\bibitem[LRW00]{LRW00}
	J. C. Lagarias, J. A. Reeds, Y. Wang, Orthonormal bases of exponentials for the $n$-cube. 
	Duke Math. J. \textbf{103} (2000), no. 1, 25--37.

\bibitem[McM80]{McM80}
	P. McMullen, Convex bodies which tile space by translation. Mathematika \textbf{27} (1980),
	no. 1, 113--121.

\bibitem[Tao04]{Tao04}
	T. Tao,  Fuglede's conjecture is false in 5 and higher dimensions. Math. Res.
	Lett. \textbf{11} (2004), no. 2--3, 251--258.


\end{thebibliography}
\end{document}